\documentclass[12pt,oneside,a4paper]{amsart}

\RequirePackage[l2tabu,orthodox]{nag}
\usepackage[all,error]{onlyamsmath}

\usepackage[english]{babel}
\usepackage[utf8]{inputenc}
\usepackage{amsmath}
\usepackage{chngcntr}
\usepackage{setspace}
\usepackage{color}
\usepackage{import}
\usepackage{amsfonts}
\usepackage{amsthm}
\usepackage{amssymb}
\usepackage{nameref}
\usepackage[shortlabels]{enumitem}
\usepackage{bbm}
\usepackage[hidelinks,colorlinks=true,urlcolor=blue,linkcolor=black,citecolor=black]{hyperref}
\usepackage{mathrsfs}
\usepackage{stmaryrd}
\usepackage[normalem]{ulem}

\allowdisplaybreaks

\newcommand*{\doi}[1]{\href{http://dx.doi.org/\detokenize{#1}}{doi}}
\counterwithin{equation}{section}

\newtheorem{theorem}[equation]{Theorem}
\newtheorem{lemma}[equation]{Lemma}
\newtheorem*{lemma*}{Lemma}

\theoremstyle{remark}
\newtheorem{remark}[equation]{Remark}

\newcommand{\RR}{\mathbb{R}}
\newcommand{\Z}{\mathbb{Z}}
\newcommand{\N}{\mathbb{N}}

\newcommand{\LL}{\mathbb{L}}

\newcommand{\prob}{\mathbb{P}}

\newcommand{\class}[1]{\langle #1 \rangle}
\newcommand{\diam}{\mathsf{diam}}

\definecolor{darkgreen}{rgb}{0.0, 0.5, 0.0}

\newcommand{\marginnote}[1]{}
\let\marginnote\leo
\def\marginnote#1{\hspace*{0em}\marginpar{\vspace*{-1.3em}\raggedleft\tiny\textcolor{red}{reminder: #1}}}

\renewcommand{\le}{\leqslant}
\renewcommand{\ge}{\geqslant}
\renewcommand{\leq}{\leqslant}
\renewcommand{\geq}{\geqslant}

\begin{document}

\title[quasi-stationary distribution of contact process]{The quasi-stationary distribution \\ of the subcritical contact process}
\date{\today}

\author{Franco Arrejoría}
\address{Universidad de Buenos Aires and IMAS-CONICET, Buenos Aires, Argentina}
\email{farrejoria@dm.uba.ar}
\author[P. Groisman]{Pablo Groisman}
\address{Universidad de Buenos Aires and IMAS-CONICET, Buenos Aires, Argentina}
\email{pgroisma@dm.uba.ar}
\author{Leonardo T. Rolla}
\address{IMAS-CONICET and NYU-ECNU Institute of Mathematical Sciences at NYU Shanghai.}
\email{leorolla@dm.uba.ar}

\begin{abstract}
We show that the quasi-stationary distribution of the subcritical contact process on $\mathbb{Z}^d$ is unique. This is in contrast with other processes which also do not come down from infinity, like stable queues and Galton-Watson, and it seems to be the first such example.
\end{abstract}
\maketitle

\section{Introduction}

The contact process models the spread of a certain infection in a population, and it is among the most studied particle systems. The configuration at time $t$ is given by a subset $\eta_t \subseteq \Z^d$ that represents the set of infected individuals.
An infected individual infects each of its neighbors at some rate $\lambda>0$ and becomes healthy at rate 1.
This system undergoes the following phase transition.
There exists a critical $0<\lambda_c<\infty$ such that, for $\lambda > \lambda_c$ the process has an invariant distribution supported on configurations with infinitely many infected sites, while for $\lambda <\lambda_c$ the only stationary distribution is the one supported on the empty configuration $\eta = \emptyset$.
See~\cite{Liggett99} for detailed background.

In this article we focus on the latter regime, which is called \emph{subcritical}.
So every distribution is attracted to $\delta_\emptyset$.
Moreover, for every initial configuration with finitely many infected sites, the process a.s.\ dies out in finite time.
In the lack of a non-trivial stationary distribution, one studies the quasi-stationary behavior of the system.
Given a process $(\zeta_t)_{t \ge 0}$ which a.s.\ reaches an absorbing configuration $\emptyset$, we say that a distribution $\nu$ is a \emph{quasi-stationary distribution} (QSD) if the process starting from $\nu$ satisfies $\prob^\nu(\zeta_t \in \cdot \,|\, \zeta_t \ne \emptyset) = \nu$.
See~\cite{ColletMartinezMartin13, MeleardVillemonais12} for detailed background on QSDs.

The process $(\eta_t)_{t\ge 0}$ described above is too rigid to have a QSD.
When conditioning on the unlikely event that $\eta_t \ne \emptyset$, typically $\eta_t$ is not localized near $\eta_0$, although it is not large in size.
In this context, the natural object to study is the \emph{contact process modulo translations}.
We say that two non-empty finite configurations are equivalent if they are translations of each other.
Let $\Lambda$ denote the quotient space resulting from this equivalence relation, and let $\langle \eta \rangle$ denote the projection of a finite configuration $\eta$ onto $\Lambda \cup \{ \emptyset \}$.
The process $(\zeta_t)_{t \geq 0}$ given by $\zeta_t:=\langle \eta_t \rangle$ is also Markovian, and it is a.s.\ absorbed at $\emptyset$.

In contrast with stationary distributions, irreducible Markov processes can have none, one or infinitely many quasi-stationary distributions.
Moreover, there is no simple classification of these three cases.
Recent work relates existence and uniqueness of QSDs to the speed at which the process comes down from infinity~\cite{CattiauxColletLambertMartinezMeleardSanMartin09, ChampagnatVillemonais16,ChampagnatVillemonais17, BansayeMeleardRichard16}.

Existence of a QSD for subcritical contact process has been shown in different papers, but the question of uniqueness remained open.
In~\cite{FerrariKestenMartinez96} it is proved that a discrete-time version of the process has a unique \emph{minimal} QSD, that is, a QSD whose mean absorption time is minimal among all the QSDs.
In~\cite{AndjelEzannoGroismanRolla15} this result was adapted to continuous time.
In~\cite{SturmSwart14} it is proved that there is a unique QSD under which the expected number of infected sites is finite.
The QSDs $\nu^*$ found in~\cite{SturmSwart14} and~\cite{AndjelEzannoGroismanRolla15} are the same as they both satisfy the Yaglom limit
$\prob^{\zeta_0} \big( \zeta_t \in \cdot\, \big| \zeta_t \neq \emptyset \big) \to \nu^*$
for arbitrary deterministic initial configuration $\zeta_0$.

Another important example of process with a.s.\ absorption is the subcritical Galton-Watson process.
This is the first class of processes for which convergence of the state at large times conditioned on non-absorption has been proved~\cite{Yaglom47}, and it is considered as a cornerstone of this theory.
Under the usual moment conditions on the offspring distribution, subcritical Galton-Watson processes have a unique QSD with finite mean, which is also the unique minimal QSD, but they have infinitely many QSDs with infinite mean and larger absorption time~\cite{SenetaVere-Jones66}.

This indicates that partial uniqueness results from~\cite{FerrariKestenMartinez96,SturmSwart14,AndjelEzannoGroismanRolla15} do not quite imply that subcritical contact process does not have other QSDs.
We also note that subcritical Galton-Watson process and subcritical contact process come down from infinity at comparable speeds.
This can be seen through the drift for the total number of (infected) individuals, which for large configurations is negative and essentially proportional to this number.

In this paper we prove that the subcritical contact process modulo translations has a unique QSD (Theorem~\ref{thm:main}).
In particular, uniqueness of the QSD is not determined by the drift of the process at infinity, and spatial constraints play a role as well (see Theorem~\ref{thm:diam}).

We finally mention that, for birth-and-death processes and a certain class of one-dimensional diffusions, it has been shown that uniqueness of the QSD is equivalent to coming down from infinity in finite time~\cite{BansayeMeleardRichard16,CattiauxColletLambertMartinezMeleardSanMartin09}.
To the best of our knowledge, subcritical contact process is the first example of a process which has a unique QSD and does not come down from infinity in finite time.

In Section~\ref{prelimiaries}, we give a precise definition of the process and describe a graphical construction.
We also recall some properties and tools for the contact process modulo translations.
In Section~\ref{uniqueness}, we prove uniqueness of the QSD for the subcritical contact process modulo translations.

\pagebreak[4]

\section{Preliminaries and main tools}\label{prelimiaries}

\subsection{Graphical representation }\label{graphical_construction}

Define the lattice $\LL^d = \Z^d + \{ \pm \frac{1}{3}e_i : i=1,\ldots,d \} $ and let $U$ be a Poisson point process in $\RR^d \times \RR$ with intensity given by $\big( \sum_{y\in \Z^d} \delta_y + \sum_{y\in \LL^d} \lambda \delta_y \big)\times \mathrm{d}t $.
Notice that $U\subseteq (\Z^d \cup \LL^d ) \times \RR$.
Let $\prob$ denote the underlying probability measure.
For nearest neighbors $x,y \in \Z^d$ we write $U^{x,y}= \{ t \in \RR:(x+\frac{y-x}{3},t) \in U \} $ and $U^x= \{ t \in\RR : (x,t) \in U\}.$ For $t\in U^{x,y}$ we say that there is an \emph{infection arrow} from $x$ to $y$ at time $t$.

Given two space-time points $(y,s)$ and $(x,t)$, we define a \emph{path from} $(y,s)$ \emph{to} $(x,t)$ as a finite sequence $(x_0,t_0),\ldots, (x_k,t_k)$ with $x_0= y, x_k=x, s=t_0\leq t_1 \leq \cdots \leq t_k=t$ with the following property.
For each $i=1,\ldots, k$, the $i$-th segment $[(x_{i-1}, t_{i-1}), (x_i,t_i)]$ is \emph{vertical}, that is, $x_i=x_{i-1}$, if $i$ is odd, or \emph{horizontal}, that is, $|| x_i-x_{i-1} ||_{1}= 1$ and $t_i=t_{i-1}$, if $i$ is even.
Horizontal segments are also referred to as \emph{jumps}.
If all horizontal segments satisfy $t_i=t_{i-1}\in U^{x_{i-1},x_i}$ then such path is also called a $\lambda$\emph{-path}.
If, in addition, all vertical segments satisfy $(t_{i-1},t_i]\cap U^{x_i}=\emptyset$ we call it an \emph{open path} from $(y,s)$ to $(x,t)$.
Existence of an open path from $(y,s)$ to $(x,t)$ is denoted by $(y,s)\leadsto (x,t)$.
For two sets $C,D\subseteq \Z^d\times \RR$, we use $C \leadsto D$ to denote the event that $(y,s) \leadsto (x,t)$ for some $(y,s) \in C, (x,t) \in D$.
We denote by $L_t$ the set $\Z^d\times \{ t \} $.

\subsection{Definition of the processes}

Given $0 \leq s \leq t$ and $A \subseteq \Z^d$, define the set of sites at time $t$ reachable from $A$ at time $s$ by
\[
\mathscr{C}_{s,t}^A = \big\{ z\in\Z^d : (x,s) \leadsto (z,t) \text{ for some } x \in A \big\}
.
\]
For $\eta_0\subseteq \Z^d$ we define 
\begin{align}
\eta_{t} = \mathscr{C}_{0,t}^{\eta_0}
, \quad t \geq 0.
\label{contactprocess}
\end{align}
We use $(\eta_t)_{t\geq 0}$ for the process defined in (\ref{contactprocess}), so it is the contact process with parameter $\lambda$ and initial configuration $\eta_0$.
By enlarging the probability space we can suppose that $\eta_0$ is random, distributed as $\nu$ and independent of $\omega$.
We denote the corresponding probability space by $\prob^\nu$, or $\prob^{\eta_0}$ if $\nu$ is supported on a deterministic $\eta_0 \subseteq \Z^d$.

Considering the lexicographic order on $\Z^d$, which is a translation invariant total order, every $\zeta \in \Lambda$ has a canonical representative $\eta \subseteq \Z^d$ by asking $\mathbf{0} \in \eta$ to be its minimal element.
From now on we identify configurations and measures on $\Lambda$ with their images under the choice of canonical representatives.
The process $(\zeta_t)_{t \geq 0}$ given by $\zeta_t:=\langle \eta_t \rangle $ is the \emph{contact process modulo translations} with initial condition $\zeta_0 = \langle \eta_0 \rangle$.
Both processes have the Markov property and $\emptyset$ as an absorbing state.

It is known that there exits a QSD $\nu^*$ on $\Lambda$ such that,
for every $\zeta_0$ and $\zeta \in \Lambda$ fixed,
\begin{align}
\label{eq:yaglom}
\prob^{\zeta_0}
( \zeta_t = \zeta | \tau>t ) \to \nu^*(\zeta)
\quad
\text{ as } t \to \infty
.
\end{align}
See for instance~\cite[Proposition~3.2]{AndjelEzannoGroismanRolla15}.
Moreover, by the Markov property the absorption time $\tau := \inf \{ t\geq 0 : \zeta_t = \emptyset \}$ satisfies 
\(
\prob^{\nu^*} (\tau > t) = e^{-\alpha t}
\label{absortion_rate}
\)
for some absorption rate $\alpha>0$.

\subsection{Good points}
\label{good_points}

In the following, the letter $\beta$ denotes an arbitrary number that will be enlarged throughout the proofs.
The product $\beta t$ means $\lfloor \beta t \rfloor $.
We say that the space-time point $(z,s)$ is a \emph{good point} if every $\lambda$-path
starting from $(z,s)$ makes fewer than $\beta t$ jumps during $[ s, s+t ] $, and we denote $G_z^s$ 
the corresponding event.
The definition of good point depends on $\beta$ and $t$ but
we omit them in the notation.
For $I\subseteq [0,+\infty)$, we let $G^I(A)$ be the event that $G^s_z$ occurs for all $s\in I$ and $z \in A$.
Defining the sets $B_r^y=\{ x \in \RR^d : \|x-y \|_\infty\leq r \}$ and $D^y_r=B_{r}^y \backslash B_{r-1}^y$, we also consider the event $\hat{G}_z^s = G^s(D^z_{2 R_t})$ that the boundary of a large square centered at $z$ is good at time $s$, and also the event $\tilde{G}_z^s = {G}_z^s \cap \hat{G}_z^s = G^s(D^z_{2 R_t} \cup \{z\})$.
Given $\rho <\infty$,
for $\beta$ and $t$ large enough one has~\cite[Lemma~2.10]{AndjelEzannoGroismanRolla15}
\begin{equation}
\label{eq:zerogood}
\prob (G_\mathbf{0}^0) \geq 1 -e^{-\rho t }
.
\end{equation}

\section{Uniqueness of the quasi-stationary distribution}
\label{uniqueness}

In this section we prove the following.

\begin{theorem}
\label{thm:main}
The subcritical contact process modulo translations has a unique QSD.
\end{theorem}

The core of the proof of Theorem~\ref{thm:main} consists in letting the spatial constraint of the model manifest itself in terms of the diameter of $\zeta_t$, denoted $\diam (\zeta_t)$.
Going back to the limit~\eqref{eq:yaglom}, note that such limit is for $\zeta_0$ fixed, and it is of course not uniform on $\zeta_0$.
Indeed, for each fixed $t$ the approximation will break down if we choose $\zeta_0$ large.
In order to prove uniqueness, we extend~\eqref{eq:yaglom} to the following limit.

\begin{theorem}
\label{thm:diam}
Let $R_t = e^{\sqrt{t}}$.
For every configuration $\zeta$,
\begin{align*}
\prob^{\zeta_0} \big( \zeta_t= \zeta \,\big|\, \tau > t \big) \, - \,
\nu^*(\zeta)
\times
\prob^{\zeta_0} \big( \diam (\zeta_t) < R_t\,\big|\, \tau > t \big)
\to 0
\end{align*}
as $t \to \infty$,
uniformly on $\zeta_0$.
\end{theorem}

\begin{remark}
The choice of $R_t = e^{\sqrt{t}}$ is rather arbitrary.
The proof given below works for any sequence $t \ll R_t \ll e^{\frac{\alpha - \varepsilon}{d}t}$. It can be refined to work for $1 \ll R_t \ll e^{\frac{\alpha}{d}t}$.
Theorem~\ref{thm:diam} is false outside this range.
\end{remark}

\begin{proof}
[Proof of Theorem~\ref{thm:main}]

Let $\nu$ be a QSD and $\zeta \in \Lambda$. By Theorem~\ref{thm:diam},
\[
\prob^{\nu} \big( \zeta_t= \zeta \,\big|\, \tau > t \big) \, - \,
\nu^*(\zeta)
\times
\prob^{\nu} \big( \diam (\zeta_t) < R_t \,\big|\, \tau > t \big)
\to 0.
\]
As $\nu$ is QSD, we have 
$\prob^{\nu} \big( \zeta_t= \zeta \,\big|\, \tau > t \big) =\nu(\zeta)$
and
\[
\prob^{\nu} \big( \diam (\zeta_t) < R_t \,\big|\, \tau > t \big)
= \nu(\{\zeta: \diam(\zeta)< R_t\})
\to 1.
\] 
Hence, $\nu(\zeta)=\nu^*(\zeta).$
\end{proof}

In the remainder of this section we prove Theorem~\ref{thm:diam}.

\smallskip

We use the notation
\[
f(\eta_0,t) \approx g(\eta_0,t)
\ \ \Leftrightarrow \ \
\limsup_{t\to\infty} \sup_{\eta_0} \big|f(\eta_0,t)-g(\eta_0,t)\big| = 0
,
\]
so Theorem~\ref{thm:diam} becomes
\begin{align}
\label{eq:approx}
\prob^{\eta_0}\big(\langle \eta_t \rangle = \zeta \,\big|\, \tau>t \big)
\approx
\nu^*(\zeta)
\times
\prob^{\eta_0} \big( \diam (\eta_t) < R_t\,\big|\, \tau > t \big)
.
\end{align}

To show~\eqref{eq:approx} we now introduce the notion of cut break point.
It is almost the same as the one introduced in~\cite{AndjelEzannoGroismanRolla15, DeshayesRolla17}.
A similar idea appeared previously in~\cite{Kuczek89} in the context of supercritical oriented percolation.

A space-time point $(y,s)$ is called a \emph{break point} if, for every $x \in B_{2R_t}^y \backslash \{ y \} $, $L_0 \not\leadsto (x,s)$.
If $\eta_t \ne \emptyset$, define
\[
X = \min \{ x \in \eta_0 : (x,0) \leadsto L_t \} \in \Z^d
\]
as the first site (for lexicographic order $\preccurlyeq$) whose infection survives up to time $t$.
We say that $(z,s)\in \Z^d \times \N$ is a \emph{cut point for $x$} if $\mathscr{C}^x_{0,s} = \{z\}$.
If it is also a break point, we call it a \emph{cut break point}.
Let $(Y,S) \in \Z^d \times \N$ be the cut break point for $x$ that appears first in time, and let $S=\infty$ if $\eta_t = \emptyset$.

\begin{lemma}
\label{lemma:goodAndSoon}
We have
\begin{equation}
\label{eq:hascutpoint}
\prob^{\eta_0}\big( S \leq \tfrac{t}{2} \,\big|\, \tau>t \big) \approx 1
\end{equation}
and, for $\beta$ large enough,
\begin{equation}
\label{eq:goodwhp}
\prob^{\eta_0}\big( \tilde{G}_S^Y \,\big|\, \tau>t \big) \approx 1
.
\end{equation}
\end{lemma}

\begin{proof}
Let $\llbracket A \rrbracket$ denote the first $\lfloor e^{2t} \rfloor$ points of $A \subseteq \Z^d$ (if $A$ has fewer than $e^{2t}$ points, take $\llbracket A \rrbracket = A$). Then,
\begin{equation}
\nonumber
\prob^{\eta_0} ( X \not\in \llbracket \eta_0 \rrbracket, \tau>t) \leq (1-e^{-t})^{e^{2t}-1}
\leq
e^{-e^t+1},
\end{equation}
similar to the proof Lemma~2.8 in~\cite{AndjelEzannoGroismanRolla15} which was proved for $d=1$.
With the above remark, the proof of~\eqref{eq:goodwhp} is the same as that of Lemma~2.8 in~\cite{AndjelEzannoGroismanRolla15}.
Now the proof of~\eqref{eq:hascutpoint} is the same as that of Lemma~3.4 in~\cite{DeshayesRolla17}.
The latter proof was also conditioning on $\eta_t \cap B_{R_t} \ne \emptyset$ which can be replaced by the above remark.
\end{proof}

Let $\mathbf{P}$ denote the probability function on $\Z^d \times \Z^d \times \N$ given by
\[
\mathbf{P}(x,y,s) = \prob^{\eta_0}\big( X=x, Y=y, S=s, \tilde{G}^s_y \,\big|\, \tilde{G}^S_Y, S<\tfrac{t}{2} \big),
\]
and $H_{x,y,s} = \{X=x, Y=y, S=s\} \cap \tilde{G}^s_y$.

The proof of Theorem~\ref{thm:diam} reduces to the following.
\begin{align*}
\prob^{\eta_0}& \big( \class{\eta_t} = \zeta \,\big|\, \tau > t \big)
\approx
\prob^{\eta_0}\big( \class{\eta_t} = \zeta \,\big|\, \tilde{G}^S_Y, S<\tfrac{t}{2} \big)
\\
& =
\int 
\prob^{\eta_0}\big( \class{\eta_t} = \zeta \,\big|\, H_{x,y,s} \big)\, \mathrm{d} \mathbf{P}
\\
& \overset{(*)}{=}
\int \prob^{\eta_0}\big( \class{\mathscr{C}_{s,t}^y} = \zeta , \diam(\eta_t)<R_t \,\big|\, H_{x,y,s} \big)\,\mathrm{d} \mathbf{P}
\\
& \overset{(\star)}{=}
\int
\prob\big( \class{\mathscr{C}_{s,t}^y} = \zeta \,\big|\, (y,s) \leadsto L_t, G_y^s \big)
\,
\prob^{\eta_0}\big( \diam(\eta_t)<R_t \,\big|\, H_{x,y,s} \big)
\,\mathrm{d} \mathbf{P}
\\
& \approx
\nu^*(\zeta)
\times
\int \prob^{\eta_0}\big( \diam(\eta_t)<R_t \,\big|\, H_{x,y,s} \big)\,  \mathrm{d} \mathbf{P}
\\
& =
\nu^*(\zeta)
\times
\prob^{\eta_0}\big( \diam(\eta_t)<R_t \,\big|\, \tilde{G}^S_Y, S<\tfrac{t}{2} \big)
\\
& \approx
\nu^*(\zeta)
\times
\prob^{\eta_0}\big( \diam(\eta_t)<R_t \,\big|\, \tau>t \big).
\end{align*}
In the above chain, the first and last equalities are just decompositions, and two of the $\approx$'s follow directly from Lemma~\ref{lemma:goodAndSoon}.
To justify the other $\approx$, note that~\eqref{eq:yaglom} and Lemma~\ref{lemma:goodAndSoon} imply that
\begin{equation}
\nonumber
\sup_{r \ge t/2}
\Big|
\prob^{\{\mathbf{0}\}}\big( \class{\eta_r} = \zeta \,\big|\, (\mathbf{0},0) \leadsto L_r , G_\mathbf{0}^0 \big)
-
\nu^*(\zeta)
\Big|
\approx 0
,
\end{equation}
which translated by $(y,s)$ gives the desired limit.

To conclude the proof, we still need to justify
that
$\overset{(*)}{=}$
and
$\overset{(\star)}{=}$
hold for all large $t$ -- large enough so that $R_t > 2 \beta t > \diam(\zeta)$.

To see why $\overset{(*)}{=}$ holds, note that on the event $H_{x,y,s}$, the only sites that can be infected at time $t$ are those in $B^y_{\beta t}$ and those in $(B^y_{2R_t-\beta t-1})^c$.
Moreover, the set of infected sites in the inner region is non-empty and equals $\mathscr{C}^y_{s,t}$.
If there are no infected sites in the outer region, then $\eta_t = \mathscr{C}^y_{s,t}$ and $\diam(\eta_t) \le 2 \beta t < R_t$, so in this case we have $\class{\eta_t} = \zeta$ if and only if $\class{\mathscr{C}^y_{s,t}} = \zeta$.
If there are infected sites in the outer region then, since $2R_t - \beta t > R_t + \beta t$, we have $\diam(\eta_t) > R_t > \diam(\zeta)$ and thus $\class{\eta_t} \ne \zeta$.

To prove $\overset{(\star)}{=}$ we first recall from the previous paragraph that, on the event $H_{x,y,x}$, $\diam(\eta_t)<R_t$ is equivalent to $\eta_t \cap (B^y_{2R_t-\beta t-1})^c = \emptyset$.
We now show that the latter event is conditionally independent of $\mathscr{C}^y_{s,t}$ given $H_{x,y,s}$.
The argument was introduced in~\cite{Ezanno12,AndjelEzannoGroismanRolla15} for $d=1$ and generalized to higher dimensions in~\cite{DeshayesRolla17}.
We reproduce it here for clarity and convenience.

The main observation is that $H_{x,y,s}$ can be written as the intersection of events which depend on disjoint regions, with the property that the first region determines the set $\mathscr{C}^y_{s,t}$ and the second determines whether $\eta_t \cap (B^y_{2R_t-\beta t-1})^c = \emptyset$ or not.
Denote $\eta_0^{\prec x} = \{z \in \eta_0 : z \preccurlyeq x, \, z\ne x\}$.

The first part of this main observation is that $H_{x,y,s}$ occurs if and only if all the events below occur:
\begin{align*}
&
G_y^s
,\\ &
(y,s) \leadsto L_t
,\\ &
\hat G_y^s
,\\ &
(x,0) \leadsto (y,s)
,\\ &
\eta_0^{\prec x} \times \{0\} \not \leadsto (y,s)
,\\ &
\eta_0^{\prec x} \times \{0\} \not \leadsto L_t
,\\ &
(x,0) \not \leadsto L_s \setminus \{(y,s)\}
,\\ &
L_0 \not\leadsto (B^y_{2 R_t} \setminus \{y\}) \times \{s\}
,\\ &
J_{x,y,s},
\end{align*}
where $J_{x,y,s}$ is the event that, for all $s'\in \{1,2,\dots,s-1\}$, either there are more than one $z\in \Z^d$ such that $(x,0) \leadsto (z,s')$, or there is a unique such $z$ but
$L_0 \mapsto (B^z_{2R_t}\setminus\{z\})\times\{s'\}$.

The 2nd, 4th and 6th events ensure that $X=x$.
The 7th and 8th events ensure that $(y,s)$ is a cut break point for $x$.
The 9th event ensures that $(y,s)$ is the earliest cut break point for $x$, which together with the event that $X=x$ ensure that $Y=y$ and $S=s$.
The 5th event is redundant once we have the 2nd and 6th, but having it listed will be helpful in the next part of the argument.

The second part of the main observation mentioned above is the following.
Let $E_{y,s} = \{(z,r) \in \RR^d \times [0,\infty): \|z-y\| \le \beta t + 1, r > s\}$.
Then the simultaneous occurrence of the 1st and 2nd events above is determined by $\omega \cap E_{y,s}$.
Moreover, when these events occur, $\mathscr{C}_{s,t}^y$ is also determined by $\omega \cap E_{y,s}$.
On the other hand, simultaneous occurrence of the 3rd--9th events above is determined by $\omega \cap E_{y,s}^c$.
Moreover, on the occurrence of these events, $\eta_t \cap (B^y_{2R_t-\beta t-1})^c$ is also determined by $\omega \cap E_{y,s}^c$.

Finally, as already argued above, on the occurrence of $H_{x,y,s}$ the events $\diam(\eta_t)<R_t$ and $\eta_t \cap D^y_{2R_t-\beta t} = \emptyset$ are equivalent.
Since $\omega \cap E_{y,s}$ and $\omega \cap E_{y,s}^c$ are independent, we can factorize the probability as in $\overset{(\star)}{=}$ for each $x,y \in \Z^d$ and $s=1,2,\dots,\lfloor \frac{t}{2}\rfloor$.
This concludes the proof of Theorem~\ref{thm:diam}.

\bibliographystyle{bib/leo}
\bibliography{bib/leo}

\end{document}